\documentclass[a4paper,11pt]{article}
\usepackage{amsfonts,amsmath,amssymb,amsthm,amsxtra,amsopn}
\usepackage{hyperref}
\usepackage{paralist}
\usepackage{booktabs}
\usepackage[all,cmtip]{xy}

\theoremstyle{plain}
\newtheorem{Theorem}{Theorem}[section]		
\newtheorem{Lemma}[Theorem]{Lemma}

\newtheorem{Corollary}[Theorem]{Corollary}
\newtheorem{proposition}[Theorem]{Proposition}

\theoremstyle{definition}

\theoremstyle{remark}
\newtheorem{remark}[Theorem]{Remark}
\newtheorem{example}[Theorem]{Example}

\numberwithin{equation}{section}

\DeclareMathOperator*{\colim}{colim}
\newcommand{\Mod}{\operatorname{\mathsf{Mod}}\nolimits}
\newcommand{\mmod}{\operatorname{\mathsf{mod}}\nolimits}
\newcommand{\End}{\operatorname{End}\nolimits}
\newcommand{\Hom}{\operatorname{Hom}\nolimits}
\newcommand{\M}{\operatorname{M}\nolimits}
\newcommand{\Fun}{\operatorname{Fun}\nolimits}
\newcommand{\Rep}{\operatorname{\mathsf{Rep}}\nolimits}
\newcommand{\add}{\operatorname{\mathsf{add}}\nolimits}

\newcommand{\op}{\mathrm{op}}
\newcommand{\ch}{\operatorname{ch}\nolimits}
\newcommand{\diag}{\operatorname{diag}\nolimits}
\newcommand{\sgn}{\operatorname{sgn}\nolimits}

\newcommand{\trace}{\mbox{trace}}
\def\frS{{\mathfrak S}}

\def\C{{\mathcal C}}

\def\M{{\mathcal M}}
\def\F{{\mathcal F}}
\def\bbZ{{\mathbb Z}}
\def\bbN{{\mathbb N}}
\def\bbQ{{\mathbb Q}}
\def\P{{\mathsf P}}
\def\Mk{{\mathsf M}_k}
\def\Sp{{\mathrm {Sp}}}
\def\i{{\underline{i}}}
\def\j{\underline{j}}

\title{The monoidal structure on strict polynomial functors}

\author{Cosima Aquilino and Rebecca Reischuk}
\date{}

\begin{document}
\maketitle

\begin{abstract}
The category of strict polynomial functors inherits an internal tensor product from the category of divided powers.
To investigate this monoidal structure, we consider the category of representations of the symmetric group $\frS_d$
which admits a tensor product coming from its Hopf algebra structure. It is classical that there exists a functor $\mathcal{F}$ from the category of strict polynomial functors
to the category of representations of the symmetric group. Our main result is that this functor $\mathcal{F}$ is monoidal. In addition we study the relations under $\F$ between projective strict
polynomial functors and permutation modules and the link to symmetric functions.
\end{abstract}

\section{Introduction}
Strict polynomial functors were first defined by Friedlander and Suslin in \cite{FS1997}, using polynomial maps of finite dimensional vector spaces over a field $k$.
They showed that the category of strict polynomial functors of a fixed degree $d$ is equivalent to the category of modules over the Schur algebra $S_k(n,d)$ whenever $n\geq d$.
These algebras, named after Issai Schur and originally used  to describe the polynomial representations of the general linear group,
have been extensively investigated.

Another description of strict polynomial functors, namely defining them as $k$-linear representations of the category of divided powers,
provides an important structure on the category of strict polynomial functors and hence on modules over the Schur algebra: The tensor product on the category of divided powers induces in 
a natural way an internal tensor product in the category of strict polynomial functors (defined in \cite{Kr2013}). 

In order to describe this tensor product more explicitly we make use of the monoidal structure on the category of representations of the symmetric group.
In fact, there exists a functor $\F$, sometimes called the Schur functor, from the category of strict polynomial functors to the category of representations of the symmetric group
which allows us to compare these monoidal structures. It turns out that the monoidal structure is preserved under the functor $\F$. This yields explicit formulae for the tensor product of certain polynomial functors.

We work over an arbitrary commutative ring $k$.
We start by recalling some basic definitions concerning strict polynomial functors and the description of the internal tensor product given in \cite{Kr2013}.
In the second section we focus our attention on representations of the symmetric group $\frS_d$. In particular, we consider the $k\frS_d$-module structure on the $d$-th tensor power of a free $k$-module $E$ and its
decomposition into permutation modules. We calculate the tensor product of permutation modules and its decomposition.

In the fourth section we show that the functor $\F$ maps certain important projective objects in the category of strict polynomial functors
to the permutation modules. For this an essential ingredient is the parametrization of morphisms of these objects given by Totaro \cite{To1997}.
Finally we  prove that $\F$ is a monoidal functor, see Theorem \ref{mainresult}.

In the last section we assume that $k$ is a field of characteristic $0$. In this case, the functor $\F$ induces an equivalence between strict polynomial functors and representations of the symmetric group. Moreover
we explain and use the connection to symmetric functions.

\section{Strict polynomial functors}
In the following we briefly recall the definition of strict polynomial functors and of the internal tensor product as described in \cite{Kr2013}.
Let $k$ be a commutative ring and denote by $\P_k$ the category of finitely generated projective $k$-modules. 
For $d\in\bbN$ and $V\in\P_k$ denote by $\Gamma^d V$ the $\frS_d$-invariant part of $V^{\otimes d}$ where the (right) action of $\frS_d$ is given by permuting the factors.
Sending a module $V$ to $\Gamma^d V$ yields a functor $\Gamma^d:\P_k\to\P_k$.

We define the \emph{category of degree $d$ divided powers $\Gamma^d\P_k$} which has the same objects as $\P_k$ and where the morphisms between two objects $V$ and $W$ are given by
\[\Hom_{\Gamma^d\P_k}(V,W):=\Gamma^d\Hom(V,W)=(\Hom(V,W)^{\otimes d})^{\frS_d}.\]

This can be identified with $\Hom(V^{\otimes d},W^{\otimes d})^{\frS_d}$ where for $\sigma\in\frS_d$, $f\in\Hom(V^{\otimes d},W^{\otimes d})$ and $v_i\in V$ the action is given by
\[f\sigma  (v_1\otimes\cdots\otimes v_d):=f((v_1\otimes\cdots\otimes v_d)\sigma^{-1})\sigma=f(v_{\sigma^{-1}(1)}\otimes\cdots\otimes v_{\sigma^{-1}(d)})\sigma.\]
In other words, the set of morphisms $\Hom_{\Gamma^d\P_k}(V,W)$ is isomorphic to the set of $\frS_d$-equivariant morphisms from $V^{\otimes d}$ to $W^{\otimes d}$. 

The \textit{category of strict polynomial functors} is the  category of $k$-linear representations of $\Gamma^d\P_k$
\[\Rep\Gamma^d_k=\Fun_k(\Gamma^d\P_k,\Mk),\]
where $\Mk$ denotes the category of $k$-modules. The morphisms between two strict polynomial functors $X,Y$ are denoted by $\Hom_{\Gamma^d_k}(X,Y).$

\paragraph*{Representable functors}
The strict polynomial functor \emph{represented by the module $V\in\Gamma^d\P_k$} is given by
\[\Gamma^{d,V}:=\Hom_{\Gamma^d\P_k}(V,-).\]
Via the Yoneda embedding the category $(\Gamma^d\P_k)^{\op}$ can be identified with the full subcategory of $\Rep\Gamma^d_k$ consisting of all representable functors.

\paragraph*{External tensor product}
For non-negative integers $d,e$ and $X\in\Rep\Gamma^d_k$ and $Y\in\Rep\Gamma^e_k$ we can form the external tensor product 
\begin{align}\label{exttensor}
 X\otimes Y\in \Rep\Gamma^{d+e}.
\end{align}
 It is given on objects by $(X\otimes Y)(V)=X(V)\otimes Y(V)$ and on morphisms via the map
\[\Gamma^{d+e}\Hom(V,W)\to\Gamma^d\Hom(V,W)\otimes\Gamma^e\Hom(V,W).\]

In particular, for positive integers $n,d$ and a composition $\lambda=(\lambda_1,\lambda_2,\dots,\lambda_n)$ of $d$ in $n$ parts, 
i.e.\ an $n$-tuple of non negative integers such that $\sum_i \lambda_i = d$, we can form representable functors 
${\Gamma^{\lambda_1,k}\in\Rep\Gamma^{\lambda_1}_k},\dots,{\Gamma^{\lambda_n,k}\in\Rep\Gamma^{\lambda_n}_k}$
and  take their tensor product to obtain a functor in $\Rep\Gamma^d_k$
\begin{align*} \Gamma^\lambda=\Gamma^{\lambda_1}\otimes\cdots\otimes\Gamma^{\lambda_n}. \end{align*}

There is (cf \cite[(2.8)]{Kr2013}) a canonical decomposition of the functor represented by $k^n$ 
\begin{align}\label{can.dec}
\Gamma^{d,k^n}=\bigoplus_{\lambda\in\Lambda(n,d)}\Gamma^\lambda,
\end{align}
where $ \Lambda(n,d)$ denotes the set of compositions of $d$ in $n$ parts.

\paragraph*{Colimits of representable functors}
We will use the fact that every object in $\Rep\Gamma^d_k$ is a colimit of representable functors. This is an analogue of a free presentation of a module over a ring, see \cite[III.7]{ML1997}.
In our situation this can be done as follows.
Let $X$  be an object in $\Rep\Gamma^d_k$ and $V\in\Gamma^d\P_k$. By the Yoneda lemma, every element $v\in X(V)$ corresponds to a natural transformation $F_v:\Gamma^{d,V}\rightarrow~X$.
Let $\C_X=\{F_v:\Gamma^{d,V}\rightarrow X \ | \ ~V\in\Gamma^d\P_k, \ v \in X(V) \}$  be the category whose objects are natural transformations $F_v$ from representable functors $\Gamma^{d,V}$ to $X$, where $V$ runs through all elements in $\Gamma^d\P_k$,
 and where a morphism between $F_v$ and $F_w$, with $v \in X(V), \ w \in X(W)$, is given by a natural transformation $\phi_{v,w}:\Gamma^{d,V}\rightarrow \Gamma^{d,W}$ such that $F_v=F_w\circ \phi_{v,w}$.
Define $\F_X:\C_X\rightarrow \Rep\Gamma^d_k$ to be the functor sending a natural transformation $F_v$ to its domain, the representable functor $\Gamma^{d,V}$.
Then $X=\colim \F_X$.

\paragraph*{Representations of Schur algebras}
For $n,d$ positive integers, the Schur algebra can be defined as 
\[S_k(n,d)=\End_{\frS_d}((k^n)^{\otimes d})=\End_{\Gamma^d_k}(\Gamma^{d,k^n})^{\op}.\]
If $n\geq d$ there is an equivalence of categories $\Rep\Gamma^d_k\cong\Mod\End_{\Gamma^d_k}(\Gamma^{d,k^n})= S_k(n,d)\Mod$ (cf. \cite[Theorem 3.2]{FS1997}). 

\subsection*{The tensor product of strict polynomial functors}
For $V, W$ in $\P_k$ denote by $V\otimes_k W$ the usual tensor product of $k$-modules. This induces a tensor product on $\Gamma^d\P_k$, the category of divided powers.
It coincides on objects with the one for $\P_k$ and on morphisms it is given via the following composite:
\[\Gamma^d\Hom(V,V')\times\Gamma^d\Hom(W,W')\rightarrow\Gamma^d(\Hom(V,V')\otimes\Hom(W,W'))\xrightarrow{\sim}\Gamma^d\Hom(V\otimes W,V'\otimes W').\]

This yields an internal tensor product on $\Rep\Gamma^d_k$, namely for representable functors $\Gamma^{d,V}$ and
$\Gamma^{d,W}$ in $\Rep\Gamma^d_k$ set
\[\Gamma^{d,V}\otimes_{\Gamma^d_k}\Gamma^{d,W}:=\Gamma^{d,V\otimes W}.\]

For arbitrary objects $X$ and $Y$ in $\Rep\Gamma^d_k$ define
\begin{align*}
\Gamma^{d,V}\otimes_{\Gamma^d_k}X&:=\colim({\Gamma^{d,V}}\otimes_{\Gamma^d_k}\F_X),\\
X\otimes_{\Gamma^d_k}Y&:=\colim(\F_X\otimes_{\Gamma^d_k}Y),
\end{align*}
where ${\Gamma^{d,V}}\otimes_{\Gamma^d_k}\F_X$ resp.\ $\F_X\otimes_{\Gamma^d_k}Y$ is the functor sending $F_v$ to ${\Gamma^{d,V}}\otimes_{\Gamma^d_k}\F_X(F_v)$ resp.\ $\F_X(F_v)~\otimes_{\Gamma^d_k}~Y$.

The tensor unit is given by \[I_{\Gamma^d_k}:=\Gamma^{d,k}\cong\Gamma^{(d)}.\]

\section{Representations of symmetric groups}\label{SymmetricGroup}

For positive integers $n,d$ let \[I(n,d) := \{ \underline{i} = (i_1\dots i_d) \mid 1\leq i_l\leq n \}\]
be the set of $d$-tuples of positive integers smaller equal than $n$.
Let $\lambda \in \Lambda(n,d) $ be a composition of $d$ in $n$ parts.
We say that $\underline{i}$ belongs to $\lambda$ (and write $\underline{i}\in\lambda$) if $\underline{i}$ has $\lambda_{l}$ entries equal to $l$.

Let $E$ be a free $k$-module of dimension $n$. For a fixed basis $ \{ e_1,\dots,e_n \} $ of $E$ the $d$-th tensor power
$E^{\otimes d}$ has a basis indexed by the set $I(n,d)$. Namely we write $e_{\i}=e_{i_1}\otimes \dots \otimes e_{i_d}$ for $\i=(i_1\dots i_d) \in I(n,d)$.

\subsubsection*{Permutation modules}
Let the symmetric group $\frS_d$ act from the right on $E^{\otimes d}$ by place permutation: 
\[(v_1 \otimes\dots\otimes v_d)\sigma = v_{\sigma(1)}\otimes \dots \otimes v_{\sigma(d)}\quad\text{ for }   \sigma \in \frS_d, \ v_1 \otimes \dots \otimes v_d \in E^{\otimes d}.\]
By linear extension of the action, $E^{\otimes d}$ becomes a $k\frS_d$-module and it decomposes as a direct sum of \emph{transitive permutation modules}
\begin{align}\label{dec_perm}
E^{\otimes d}=\bigoplus_{\lambda \in \Lambda(n,d)} M^{\lambda},
\end{align} where $M^{\lambda}$ is the $k$-span of the set $\{e_{\i} \ | \ \i \mbox{ belongs to } \lambda \}$.
Note that this set is the same as $\{e_{\i_{\lambda}}  \sigma \ | \  \sigma \in \frS_d/\frS_{\lambda} \}$ where $\i_{\lambda} = (1\dots1 ~2\dots2 \dots n\dots n)$ has
$ \lambda_l$ entries equal to $l$ and $\frS_{\lambda}$ denotes the Young subgroup $\frS_{\lambda_1} \times \dots \times \frS_{\lambda_n} \subset \frS_d$.
So we have a one to one correspondence between the elements of a basis of $M^{\lambda}$ and the elements of the group $\frS_d/\frS_{\lambda}$.

To every $\i\in\lambda$ we can associate a  dissection $d_{\i}$ of the set $ \{1,\dots,d\} $ into subsets of order $\lambda_1, \dots, \lambda_n$ as follows:
\[d_{\i}:=\{d_{\i}^1,\dots,d_{\i}^n\}\text{ with }d_{\i}^l:=\{j\mid i_j=l\}\]
That means $d_{\i}^l$ consists of all indices at which $l$ is situated in $\i$. Note that $d_{\i}^l\cap d_{\i}^{l'}=\emptyset$ for $l\neq l'$.

A permutation module $M^{\lambda}$ can be identified with the $k$-span of all dissections  $d_{\i}$ with $\i\in\lambda$, that we denote by $d_\lambda$. 
The action of $\sigma\in\frS_d$  on $d_{\i}^l$ is given by $\sigma^{-1}j$ for every element $j\in d_{\i}^l\subseteq \{1,\dots,d\} $.

\subsubsection*{Left modules}
If we denote by $\mmod k\frS_d$ the (right) $k\frS_d$-modules that are finitely generated projective over $k$, we have an equivalence of categories:
\begin{align*}\Hom_{k\frS_d}(-,k\frS_d)\colon\mmod k\frS_d^{\op}&\to k\frS_d\mmod\\
M&\mapsto\Hom_{k\frS_d}(M,k\frS_d)
\end{align*}
where the left action on $\Hom_{k\frS_d}(M,k\frS_d)$ is given by $(\pi f)\colon m\mapsto \pi\cdot f(m)$ for $\pi\in\frS_d$,\linebreak $f\in\Hom_{k\frS_d}(M,k\frS_d)$ and $m\in M$. Note that we can identify $\Hom_{k\frS_d}(E^{\otimes d},k\frS_d)$
with the left module $E^{\otimes d}$ where the action for $\pi\in\frS_d$ and $v_1 \otimes\dots\otimes v_d\in E^{\otimes d}$ is given by
\[\pi(v_1 \otimes\dots\otimes v_d):= v_{\pi^{-1}(1)}\otimes \dots \otimes v_{\pi^{-1}(d)}.\]
Namely, if we denote by $^\lambda M\subseteq E^{\otimes d}$ the left permutation module corresponding to $\lambda$, i.e.\ the module with basis
$\{\pi e_{\i_{\lambda}} :\pi \in \frS_d/\frS_{\lambda} \}$, we get the following isomorphism of left modules:
\begin{align*}
  ^\lambda M&\to \Hom_{k\frS_d}(M^\lambda,k\frS_d)\\
 \pi e_{\i_{\lambda}}&\mapsto (f_\pi:e_{\i_{\lambda}}\mapsto \pi)
\end{align*}
We will identify $^\lambda M$ and $\Hom_{k\frS_d}(M^\lambda,k\frS_d)$ via this isomorphism.

\subsection*{The tensor product of representations of symmetric groups}
The $k\frS_d$-module structure on the tensor product of two representations of $\frS_d$ is given via the Hopf algebra structure of the group algebra $k\frS_d$, i.e.\ for $M,~N\in \Mod k\frS_d$, let
$M\otimes_{k}N$ be the usual tensor product over $k$ together with the following diagonal action of $\frS_d$:
\begin{align*}
(M\otimes_{k}N)\times k\frS_d&\rightarrow (M\otimes_k N)\\
 ((m\otimes n),\pi)&\mapsto (m \pi\otimes n \pi)
\end{align*}

The tensor unit is given by $M^{(d)}\cong k$, the trivial $k\frS_d$-module.

\paragraph*{Tensoring permutation modules}
For $\lambda\in\Lambda(m,d)$ and $\mu\in\Lambda(n,d)$ define $A^\lambda_\mu$ to be the set of all $m\times n$ matrices $A=(a_{ij})$ with entries in $\bbN$ such that $\lambda_i=\sum_j a_{ij}$
and $\mu_j=\sum_ia_{ij}$.

For a field $k$ of characteristic 0, James and Kerber showed in  \cite{JK1981} how to decompose the tensor product of two permutation modules in terms of characters.
The following is an analogue for $k$ an arbitrary commutative ring:
\begin{Lemma}\label{tensorpermmod}
The tensor product of two permutation modules $M^\lambda$ and $M^\mu$ can be decomposed into permutation modules as follows:
\[M^\lambda\otimes_{k} M^\mu\cong\bigoplus_{A\in A^\lambda_\mu} M^A,\]
where $A$ is regarded as the composition $(a_{11},a_{12},\dots,a_{21},a_{22},\dots,a_{mn})$. 
\end{Lemma}
\begin{proof}
The idea of the proof is taken from \cite{JK1981}. If we take, for any composition $\nu$ of $d$, as a basis for the permutation module $M^\nu=d_\nu$ the set of all dissections $d_{\i}$ with $\i\in\nu$,
then a basis of $M^\lambda\otimes_{k} M^\mu=d_\lambda\otimes d_\mu$ is given by all products $d_{\i}\otimes d_{\j}$ with $\i\in\lambda$ and $\j\in\mu$.
 
Consider now the orbits of ${d}_{\lambda}\otimes {d}_\mu$ under the action of $\frS_d$.
Set $A^{{\i}}_{{\j}}:=(|d_{\i}^s\cap d_{\j}^t|)_{st}\in\M_{m\times n}(\bbN)$.
Note that $A^{{\i}}_{{\j}}\in A^\lambda_\mu$ for all $\i\in\lambda$ and $\j\in\mu$.
Two basis elements $d_{\i}\otimes d_{\j}$ and $d_{\i'}\otimes d_{\j'}$ belong to the same orbit if and only if
$A^{{\i}}_{{\j}}=A^{{\i'}}_{{\j'}}$. 
Thus, we can decompose the $k\frS_d$-module ${d}_{\lambda}\otimes {d}_\mu$ into a direct sum of submodules
\[{d}_{\lambda}\otimes {d}_\mu=\bigoplus_{A\in A^\lambda_\mu}({d}_\lambda\otimes{d}_\mu)_{A},\]
where $({d}_\lambda\otimes{d}_\mu)_{A}\subseteq {d}_\lambda\otimes{d}_\mu$ is spanned by all $d_{\i}\otimes d_{\j}$ such that $A=A^{{\i}}_{{\j}}$.
But $({d}_\lambda\otimes{d}_\mu)_{A}\cong {d}_A$ as $k\frS_d$-modules, where the isomorphism is given by 
\[\{d_{\i}^1,\dots,d_{\i}^m\}\otimes \{d_{\j}^1,\dots,d_{\j}^n\}\mapsto \{d_{\i}^1\cap d_{\j}^1,d_{\i}^1\cap d_{\j}^2,\dots,d_{\i}^2\cap d_{\j}^1,\dots,d_{\i}^m\cap d_{\j}^n\}.\qedhere\]
\end{proof}

\begin{remark}
 In the same way, it is possible to tensor left permutation modules. Namely for $^\lambda M=\Hom_{k\frS_d}(M^\lambda,k\frS_d)$ and $^\mu M=\Hom_{k\frS_d}(M^\mu,k\frS_d)$, we get
 \[^\lambda M\otimes_{k}~ ^\mu M\cong\bigoplus_{A\in A^\lambda_\mu}~ ^A M.\]
 For left modules, the tensor unit is given by $I_{k\frS_d}:={^{(d)}M}$.
\end{remark}

\begin{example}
 Let $\lambda=(3,1)\in \Lambda(2,4)$ and $\mu=(2,1,1)\in\Lambda(3,4)$. Then $A^\lambda_\mu$ consist of the following matrices:
\[ \begin{pmatrix}
    2&1&0\\
    0&0&1\\
   \end{pmatrix},
   \begin{pmatrix}
    2&0&1\\
    0&1&0\\
   \end{pmatrix},
      \begin{pmatrix}
    1&1&1\\
    1&0&0\\
   \end{pmatrix}.
\]
Let now $\i=(1112) \in \lambda$ and $\j=(1312) \in \mu$. Then $d_{\i}=\{\{123\},\{4\}\}$ and $d_{\j}=\{\{13\},\{4\},\{2\}\}$. The orbit of  $d_{\i}\otimes d_{\j}$ consists of the elements
\begin{align*}
 \{(d_{\i}\pi \otimes d_{\j}\pi ) \ | \ \pi\in\frS_4\}=&\{\{123\},\{4\}\}\otimes\{\{13\},\{4\},\{2\}\},\\
 &\{\{213\},\{4\}\}\otimes\{\{23\},\{4\},\{1\}\},\\
&\{\{132\},\{4\}\}\otimes\{\{12\},\{4\},\{3\}\},\\
&\{\{231\},\{4\}\}\otimes\{\{12\},\{4\},\{3\}\},\\
&\{\{312\},\{4\}\}\otimes\{\{23\},\{4\},\{1\}\},\dots\\
\end{align*}
and $A^{{\i}}_{{\j}}=\begin{pmatrix}
                            |\{123\}\cap\{23\}|& |\{123\}\cap\{4\}|& |\{123\}\cap\{1\}|\\
                            |\{4\}\cap\{23\}|& |\{4\}\cap\{4\}|& |\{4\}\cap\{1\}|\\
                           \end{pmatrix}=
                           \begin{pmatrix}
                            2& 0&1 \\
                            0&1&0\\
                           \end{pmatrix}.$\\
Recall that $A^{{\i}}_{{\j}}=A^{{\i}\pi}_{{\j}\pi}$ for all $\pi\in\frS_d$ and that $(d_\lambda\otimes{d}_\mu)_{A^{{\i}}_{{\j}}}\cong{d}_{A^{{\i}}_{{\j}}}$. 
So we get \[({d}_\lambda\otimes{d}_\mu)_{A^{{\i}}_{{\j}}}\cong M^{A^{{\i}}_{{\j}}}=M^{(2,0,1,0,1,0)}\cong M^{(2,1,1)}.\] 
There are two more orbits that can be obtained by taking, for example, the elements
\begin{eqnarray*}
 &\i'=(1112)\text{ and }\j'=(1213)~\text{ resp. }&\i''=(1112)\text{ and }\j''=(1321),\\
 \end{eqnarray*}
which span the submodules $M^{(2,1,0,0,0,1)}\cong M^{(2,1,1)}$ resp. $M^{(1,1,1,1,0,0)}\cong M^{(1,1,1,1)}$.
All in all we get
\[M^{(3,1)}\otimes M^{(2,1,1)}\cong M^{(2,1,1)}\oplus M^{(2,1,1)}\oplus M^{(1,1,1,1)}=2\cdot M^{(2,1,1)}\oplus M^{(1,1,1,1)}.\]
\end{example}

\section{From strict polynomial functors to representations of the symmetric group}
The set of matrices $A^\lambda_\mu\subset \M_{m\times n}(\bbN)$ defined in the previous section also plays  an important role in describing the set 
of morphisms  $\Hom_{\Gamma^d_k}(\Gamma^\lambda,\Gamma^\mu)$ in $\Rep\Gamma^d_k$. This has been done by Totaro in \cite[p. 8]{To1997} and is also used in
\cite[Lemma 4.3]{Kr2014}. It yields the following
\begin{Lemma} 
 Let $\lambda\in\Lambda(n,d)$, $\mu\in\Lambda(m,d)$ be compositions of $d$. The elements of  $ A^\lambda_\mu$ give a natural $k$-basis of $\Hom_{\Gamma^d_k}(\Gamma^\lambda,\Gamma^\mu)$.\qed
\end{Lemma}
Using \ref{can.dec} and \ref{dec_perm} one obtains the following decompositions
\begin{align*}
\End_{\Gamma^d_k}(\Gamma^{d,k^n})&=\bigoplus\limits_{\lambda,\mu\in\Lambda(n,d)} \Hom_{\Gamma^d_k}(\Gamma^\lambda,\Gamma^\mu)\\
S_k(n,d) = \End_{k\frS_d}(E^{\otimes d})&=\bigoplus\limits_{\lambda,\mu\in\Lambda(n,d)}\Hom_{k\frS_d}(M^\mu,M^\lambda)
\end{align*}
and using the description of a basis of the Schur algebra (e.g.\ in \cite{Green}), one gets even more:
\begin{Lemma}
  Let $\lambda,~\mu\in\Lambda(n,d)$ be compositions of $d$. Then 
  \[\Hom_{\Gamma^d_k}(\Gamma^\lambda,\Gamma^\mu)\cong \Hom_{k\frS_d}(M^\mu,M^\lambda)\]
  where the basis elements of $S_k(n,d)$ contained in $\Hom_{k\frS_d}(M^\mu,M^\lambda)$ can be identified with the elements of $A^\lambda_\mu$. \qed
\end{Lemma}
For $\omega=(1,\dots,1)\in\Lambda(d,d)$ the composition of $d$ consisting of $d$ times $1$, this yields in particular
\[\Hom_{\Gamma^d_k}(\Gamma^\omega,\Gamma^\mu)\cong \Hom_{k\frS_d}(M^\mu,k\frS_d)\]
\[\End_{\Gamma^d_k}(\Gamma^\omega)\cong \End_{k\frS_d}(M^\omega)^\op\cong \End_{k\frS_d}(k\frS_d)^\op\cong k\frS_d^\op,\]
where we use the identification $M^\omega\cong k\frS_d$.
Thus we get a functor $\F$ from the category of strict polynomial functors to the category of 
representations of the symmetric group:
\[\F=\Hom_{\Gamma^d_k}(\Gamma^\omega,-)\colon\Rep\Gamma^d_k\rightarrow\Mod k\frS_d^\op=k\frS_d \Mod\]
The functors $\Gamma^\lambda$ are mapped under $\F$ to the permutation modules,\ i.e. 
\[\F(\Gamma^\lambda)=\Hom_{\Gamma^d_k}(\Gamma^\omega,\Gamma^\lambda)\cong\Hom_{k\frS_d}(M^\lambda,M^\omega)\cong\Hom_{k\frS_d}(M^\lambda,k\frS_d)={^\lambda M}\]

Note that, in particular, the representable functor $\Gamma^{d,k^n}=\bigoplus_{\lambda\in\Lambda(n,d)}\Gamma^\lambda$ is mapped to
$\Hom_{k\frS_d}(E^{\otimes d},k\frS_d)=\bigoplus_{\lambda\in\Lambda(n,d)}{^\lambda M}$.

\paragraph*{An equivalence of categories}
Let $\Gamma=\{\Gamma^\lambda\}_{\lambda\in\Lambda(n,d)}$ and $M=\{{^\lambda M}\}_{\lambda\in\Lambda(n,d)}$. Denote by $\add\Gamma$ the full subcategory of $\Rep\Gamma^d_k$ whose 
objects are direct summands  of finite direct sums of $\Gamma^\lambda$. Define $\add M$ similarly as a subcategory of $k\frS_d\Mod$.

\begin{Lemma}\label{equivalence}
The functor $\F=\Hom_{\Gamma^d_k}(\Gamma^\omega,-)$ induces an equivalence of categories between $\add\Gamma$ and $\add M$.
\end{Lemma}

\begin{proof}
Since $\Gamma^{d,k^n}=\bigoplus_{\lambda\in\Lambda(n,d)}\Gamma^\lambda$ we have $\add\Gamma=\add \Gamma^{d,k^n}$. Similarly one can see that $\add M=\add \Hom_{k\frS_d}(E^{\otimes d},k\frS_d)$.
Thus we get the following commutative diagram:
\[\xymatrixcolsep{5pc}\xymatrix{\Rep\Gamma^d_k\ar[r]^{\F=\Hom_{\Gamma^d_k}(\Gamma^\omega,-)} &k\frS_d\Mod \\
	    **[l]\add\Gamma^{d,k^n}=\add\Gamma\ar[r]^{\F|_{\add\Gamma}} \ar@{^{(}->}[u]&**[r]\add M=\add \Hom_{k\frS_d}(E^{\otimes d},k\frS_d)\ar@{^{(}->}[u]\\
	    }\]
The object $\Gamma^{d,k^n}$ is mapped under $\F$ to $\Hom_{k\frS_d}(E^{\otimes d},k\frS_d)$. 

For the morphisms $\F$ induces the following isomorphism:
\begin{align*}\Hom_{\Gamma^d_k}(\Gamma^\lambda,\Gamma^\mu)&\rightarrow\Hom_{k\frS_d}(\Hom_{\Gamma^d_k}(\Gamma^\omega,\Gamma^\lambda),\Hom_{\Gamma^d_k}(\Gamma^\omega,\Gamma^\mu))\\
&\cong\Hom_{k\frS_d}(\Hom_{k\frS_d}(M^\lambda,k\frS_d),\Hom_{k\frS_d}(M^\mu,k\frS_d))=\Hom_{k\frS_d}({^\lambda M},{^\mu M})\qedhere\end{align*}
\end{proof}

If we do not restrict  to the subcategories  $\add\Gamma$ and $\add M$ the functor $\F=\Hom_{\Gamma^d_k}(\Gamma^\omega,-)$ is not an equivalence in general. 
Schur proved that it is for $k$ a field of characteristic $0$. But for example if $k$ is a field of positive characteristic $p$, 
the categories $\Rep\Gamma^p_k$ and $k\frS_p\Mod $ are not equivalent.

\paragraph*{The monoidal structure}
Independently of any assumption on the commutative ring $k$ we have the following main result of this work:
\begin{Theorem} \label{mainresult}
 The functor \[\F=\Hom_{\Gamma^d_k}(\Gamma^\omega,-)\colon\Rep\Gamma^d_k\rightarrow  k\frS_d\Mod\] preserves the monoidal structure defined on strict polynomial functors,
 i.e.
 \begin{align}\label{eq:iso}\Hom_{\Gamma^d_k}(\Gamma^\omega,X\otimes_{\Gamma^d_k}Y)\cong\Hom_{\Gamma^d_k}(\Gamma^\omega,X)\otimes_{k}\Hom_{\Gamma^d_k}(\Gamma^\omega,Y)\end{align}
 for all $X$ and $Y$ in $\Rep\Gamma^d_k$ and 
 \[\Hom_{\Gamma^d_k}(\Gamma^\omega,I_{\Gamma^d_k})=I_{k\frS_d}.\]
\end{Theorem}

\begin{proof}
As observed in the first section, every functor $X$ in $\Rep\Gamma^d_k$ is a colimit of representable functors.
One can show that we obtain the same if we only take the colimit with respect to those functors that are represented by finitely generated free modules. 

Moreover the functor $\Hom_{\Gamma^d_k}(\Gamma^\omega,-)$ preserves colimits, since it has a right adjoint.
Thus it is enough to show the isomorphism \eqref{eq:iso} for functors represented by free modules. 
Let $V=k^n$ and $W=k^m$ for some non-negative integers $n$ and $m$.
Using the definition of the internal tensor product and the canonical decomposition \eqref{can.dec} we get
\[\Gamma^{d,k^n}\otimes_{\Gamma^d_k}\Gamma^{d,k^m}=\Gamma^{d,k^n\otimes k^m}\cong\Gamma^{d,k^{n\cdot m}}=\bigoplus_{\nu\in\Lambda(n\cdot m,d)}\Gamma^\nu.\]
Writing down the entries of $\nu\in\Lambda(n\cdot m,d)$ in an $n\times m$ matrix, we obtain a bijection between the set $\Lambda(n\cdot m,d)$ and the set of all $n\times m$ matrices
with entries in $\bbN$ such that the sum of all entries is $d$. Every such matrix $A=(a_{ij})$ defines a couple $(\lambda,\mu)$ with $\lambda\in\Lambda(n,d)$ and $\mu\in\Lambda(m,d)$ where
$\lambda_i$ is given by $\sum_{j}a_{ij}$ and $\mu_j$ is given by $\sum_ia_{ij}$, so that $A\in A^\lambda_\mu$. 
All in all we get a bijection of sets
\[\Lambda(n\cdot m,d)\longleftrightarrow \{A\in\M_{n\times m}(\bbN)\mid \sum_{st}a_{st}=d\}\longleftrightarrow \bigcup_{\substack{\lambda\in\Lambda(n,d)\\\mu\in\Lambda(m,d)}}A^\lambda_\mu\]
and thus the following decomposition
\[\Gamma^{d,k^n}\otimes_{\Gamma^d_k}\Gamma^{d,k^m}\cong\bigoplus_{\nu\in\Lambda(n\cdot m,d)}\Gamma^\nu=\bigoplus_{\substack{\lambda\in\Lambda(n,d)\\\mu\in\Lambda(m,d)}}~\bigoplus_{A\in A^\lambda_\mu}\Gamma^A,\]
where the matrix $A=(a_{ij})$ is seen as the composition $(a_{11},a_{12},\dots,a_{21},a_{22},\dots,a_{mn})$.

Finally this yields
\begin{align*}\Hom_{\Gamma^d_k}(\Gamma^\omega,\Gamma^{d,k^n}\otimes_{\Gamma^d_k}\Gamma^{d,k^m})&\cong\Hom_{\Gamma^d_k}(\Gamma^\omega,\bigoplus_{\substack{\lambda\in\Lambda(n,d)\\\mu\in\Lambda(m,d)}}~\bigoplus_{A\in A^\lambda_\mu}\Gamma^A)\\
&\cong\bigoplus_{\substack{\lambda\in\Lambda(n,d)\\\mu\in\Lambda(m,d)}}~\bigoplus_{A\in A^\lambda_\mu}\Hom_{\Gamma^d_k}(\Gamma^\omega,\Gamma^A)\\
&\cong\bigoplus_{\substack{\lambda\in\Lambda(n,d)\\\mu\in\Lambda(m,d)}}~\bigoplus_{A\in A^\lambda_\mu}{^A M}\\
&\overset{(*)}{\cong}\bigoplus_{\substack{\lambda\in\Lambda(n,d)\\\mu\in\Lambda(m,d)}}{^\lambda M}\otimes_{k}{^\mu M}\\
&\cong\left(\bigoplus_{\lambda\in\Lambda(n,d)}{^\lambda M}\right)\otimes_{k}\left(\bigoplus_{\mu\in\Lambda(m,d)}{^\mu M}\right)\\
&\cong\left(\Hom_{\Gamma^d_k}(\Gamma^\omega,\Gamma^{d,k^n})\right)\otimes_{k}\left(\Hom_{\Gamma^d_k}(\Gamma^\omega,\Gamma^{d,k^m})\right),\end{align*}
where $(*)$ is due to Lemma \ref{tensorpermmod}.

For the respective tensor units we get:
\[\Hom_{\Gamma^d_k}(\Gamma^\omega,I_{\Gamma^d_k})=\Hom_{\Gamma^d_k}(\Gamma^\omega,\Gamma^d)\cong \Hom_{k\frS_d}(M^{(d)},k\frS_d)={^{(d)}M}=I_{k\frS_d}.\]

The naturality of the coherence maps is obtained using that $\F(\Gamma^\lambda)=\Hom_{k\frS_d}(M^\lambda,k\frS_d)$ and the naturality of $\Hom_{k\frS_d}(-,k\frS_d)$. 
\end{proof}

\begin{Corollary}
 The tensor product of $\Gamma^\lambda$ and $\Gamma^\mu$ can be decomposed by the same rule as the tensor product of $M^\lambda$ and $M^\mu$, namely
 \[\Gamma^\lambda\otimes_{\Gamma^d_k}\Gamma^\mu\cong\bigoplus_{A\in A^\lambda_\mu}\Gamma^A\]
\end{Corollary}

\begin{proof}
 We have that
 \[\bigoplus_{\nu\in\Lambda(n\cdot m, d)}\Gamma^\nu=\Gamma^{d,k^{n\cdot m}}\cong\Gamma^{d,k^n}\otimes_{\Gamma^d_k}\Gamma^{d,k^m}=\left(\bigoplus_{\lambda\in\Lambda(n,d)}\Gamma^\lambda\right)\otimes_{\Gamma^d_k}\left(\bigoplus_{\mu\in\Lambda(m,d)}\Gamma^\mu\right)\cong\bigoplus_{\substack{\lambda\in\Lambda(n,d)\\\mu\in\Lambda(m,d)}}\Gamma^\lambda\otimes_{\Gamma^d_k}\Gamma^\mu\]
 and thus $\Gamma^\lambda\otimes_{\Gamma^d_k}\Gamma^\mu$ belongs to $\add\Gamma$. The equivalence in  Lemma \ref{equivalence} yields the stated decomposition.
\end{proof}

\section{Relation to symmetric functions in characteristic 0}
Assume that $k$ is a field of characteristic $0$. In this case, the categories $\Rep\Gamma^d_k$ and $ k\frS_d\Mod$ are semisimple
and the functor $\F=\Hom_{\Gamma^d_k}(\Gamma^\omega,-)$ induces an equivalence. The simple objects in $\Rep\Gamma^d_k$ are given by the so called Schur functors and are 
mapped via $\F$ to the Specht modules i.e., the simple $k\frS_d$ modules. Since $ k\frS_d\Mod$ is semisimple we can identify every simple module with
its character. As Macdonald explains in \cite{Macd1995} there is an isometric isomorphism from the ring of irreducible characters of $k\frS_d$, for all $d\geq 0$, to the ring of symmetric functions.
We will explain the various correspondences.

\subsection{Ring of symmetric functions}
Following \cite{Macd1995} denote by $\Lambda$ the $\bbZ$-graded ring of symmetric functions. We recall various definitions and refer to \cite{Macd1995} for more details.
For any sequence of natural numbers $\alpha=(\alpha_1,\alpha_2,\dots,\alpha_n)$ the \emph{monomial symmetric functions} are defined by
\[m_\alpha(x_1,\dots,x_n):=\sum_{\beta\sim\alpha} x^\beta=\sum_{\beta\sim\alpha} x_1^{\beta_1}\cdots x_n^{\beta^n}\]
where the sum is taken over all different permutations $\beta$ of $\alpha$. The distinct monomial symmetric functions form a $\bbZ$-basis of $\Lambda$. There are many other bases,
in our context some of them  are of special interest. We describe them briefly.
 
\begin{itemize}
\item[\bf Elementary symmetric functions] For any natural number $n$ define the $n$-th elementary symmetric function \[e_n:=m_{(1^n)}\]
For any set of variables this is the sum of all products of $n$ distinct variables, i.e.  \[e_n(x)=\sum_{i_1<i_2<\dots<i_n} x_{i_1}\dots x_{i_n}.\]
If $ \lambda=(\lambda_1, \dots, \lambda_n) $ is any sequence of natural numbers, set $ e_\lambda:=e_{\lambda_1}\cdots e_{\lambda_n} $.
An ordered composition of $d$, i.e. $\lambda=(\lambda_1, \dots, \lambda_n)$ with $\lambda_1 \geq \dots\geq \lambda_n \geq 0$ is usually called a \emph{partition} of $d$.
The set $ \{ e_{\lambda } \ | \ \lambda \mbox{ partition of } d \} $ is a basis for the symmetric functions of degree $d$. If we consider all partitions of all non-negative integers $d$ we
obtain a $\bbZ$-basis of the ring $\Lambda$. 
\item[\bf Complete symmetric functions] For any natural number $n$ define the $n$-th complete symmetric function
\[h_n:=\sum_{|\alpha|=n} m_\alpha.\]
For any set of variables this is the sum of all monomials of total degree $n$, i.e.  \[h_n(x)=\sum_{|\alpha|=n}~\sum_{\beta\sim\alpha} x_1^{\beta_1}\cdots x_n^{\beta^n},\]
where the second sum is taken over all distinct permutations of $\alpha$. Again, the set of $h_\lambda:=h_{\lambda_1}\cdots h_{\lambda_n}, $ for all partitions $\lambda$, forms a $\bbZ$-basis of the ring $\Lambda$.
\item[\bf Schur functions] With $\lambda'\in\Lambda(n',d)$ the conjugate partition of $\lambda\in\Lambda(n,d)$,  the Schur function $s_\lambda$ can be defined as follows:
\[s_\lambda:=\det(h_{\lambda_i-i+j})_{1\leq i,j\leq n}=\det(e_{\lambda'_i-i+j})_{1\leq i,j\leq n'}.\] The set of Schur functions corresponding to all partitions is another possible $\bbZ$-basis of the ring $\Lambda$.
\item[\bf Power sum] For any natural number $n$ define the $n$-th power sum
\[p_n:=m_{(r)}.\] Then, $ \{ p_\lambda:=p_{\lambda_1}\cdots p_{\lambda_n} \ | \ \lambda \mbox{ a partition} \}$ is a $\bbQ$-basis of $\Lambda_\bbQ=\Lambda\otimes_{\bbZ}\bbQ$. 
\end{itemize}

\paragraph*{Relations between symmetric functions}
There are various relations between the symmetric functions defined above given by the so called Kostka numbers, cf. \cite[6.~Table~I]{Macd1995}. The most interesting ones for our purpose are
\begin{align}\label{symmfunc} h_\lambda=\sum_{\mu\in\Lambda(n,d)}K_{\mu\lambda}s_\mu\qquad h_{(d)}=h_d=s_{(d)},  \end{align}
where  $K_{\lambda\mu}$ is the number of tableaux of shape $\lambda$ and content $\mu$. In particular $K_{\lambda\mu}=1$ if $\lambda=\mu$ and $K_{\lambda\mu}\neq 0$ if and only if $\lambda\geq\mu$ \ with respect to
the dominance order.
                                                                                                                  
\paragraph*{Scalar product} One can define a scalar product on $\Lambda$ by requiring that the bases $\{h_\lambda\}$ and $\{m_\lambda\}$ should be dual to each other, i.e.
\[\left<h_\lambda,m_\mu\right>=\delta_{\lambda\mu}.\]
This implies that
\[\left<s_\lambda,s_\mu\right>=\delta_{\lambda\mu}.\]

\subsection{Representations of the symmetric group}
For a field of characteristic $0$, the category $k\frS_d\Mod $ is semisimple.  The simple modules are given by the Specht modules defined as follows:
For a partition $\lambda$ of $d$, let $\i_\lambda$ be the $d$-tuple where $1$ occurs $\lambda_1$ times, $2$ occurs $\lambda_2$ times and so on, i.e.
$\i_\lambda'=(12\dots n~\dots~1\dots n-1~\dots)\in\lambda$.
\begin{example}
Let $\lambda=(4,3,3,1)\in\Lambda(4,11)$. Then $\i_\lambda'=(1234~123~123~1)$.
\end{example}
This yields the element $e_{\i_\lambda'}=e_1\otimes e_2\otimes\dots \otimes e_n\otimes e_1\otimes e_2\dots\in {^\lambda M}$. 
Take the Young subgroup $\frS_{\lambda'}=\frS_{\lambda_1'}\times\dots\times \frS_{\lambda_n'}\subset \frS_d$
and define $v_\lambda:=\sum_{\sigma\in\frS_{\lambda'}}\sgn(\sigma)\sigma e_{\i_\lambda'}$. 
Then the Specht module $\Sp(\lambda)\subset M^\lambda$ is generated by $v_\lambda$, i.e.
\[\Sp(\lambda)= k\frS_d\,v_\lambda.\]
\begin{remark}
 If one regards the elements of ${^\lambda M}$ as tabloids defined by tableaux with filling $\{1,\dots,d\}$, one can define an element $v_T$ for every tableau $T$.
 Then a $k$-basis of $\Sp(\lambda)$ is given by those $v_T$ coming
 from a standard tableau, i.e. where the entries are strictly increasing along each row and each column (cf. {\cite[7.2~Proposition~2]{Fu1997}}).
\end{remark}
The relation between the permutation modules and the Specht modules is also given via Kostka numbers. Since $k\frS_d$ is semisimple, each module can be decomposed
into simple ones. Hence, every permutation module is the sum of some Specht modules, namely
\begin{align}\label{decomposition} {^\mu M}=\bigoplus_{\lambda\in\Lambda(n,d)}K_{\lambda\mu}\Sp(\lambda).\end{align}
 In particular
\[{^{(d)}M}=\Sp((d)).\]

Let $R^d$ be the free $\bbZ$-module generated by the irreducible characters of $\frS_d$. One can define a graded ring structure on $R=\bigoplus_{d\geq 0}R^d$ where the
multiplication is given by inducing the character $\phi\times\psi$ of $\frS_d\times\frS_e$ to a character of $\frS_{d+e}$. In terms of modules this implies that
\[{^\lambda M}={^{\lambda_1}M}\otimes\cdots\otimes {^{\lambda_n}M}=\Sp(\lambda_1)\otimes\cdots\otimes\Sp(\lambda_n).\]

A scalar product is given by the usual scalar product of functions on a group, namely 
\[\left<\phi,\psi\right>_{\frS_d}=\frac{1}{d!}\sum_{\pi\in\frS_d}\phi(\pi)\psi(\pi^{-1}).\]

\paragraph*{Tensor product}
The tensor product  of two Specht modules can be described in terms of so called Kronecker coefficients $g^\nu_{\lambda\mu}$:
\[\Sp(\lambda)\otimes_k \Sp(\mu)\cong\bigoplus_{\nu\in\Lambda(n,d)}g^\nu_{\lambda\mu}\Sp(\nu)\]

\subsection{Strict polynomial functors}
The simple objects of $\Rep\Gamma^d_k$  are given by the Schur functors $S_\lambda$ defined in \cite[II.1]{ABW1982} (c.f. also \cite{Kr2013})
as follows: Let $\sigma_\lambda$ be the permutation of $\frS_d$ defined on $r=\lambda_1+\cdots+\lambda_{i-1}+j$ by
\[\sigma(r)=\sigma_\lambda(\lambda_1+\cdots+\lambda_{i-1}+j)=\lambda_1'+\cdots+\lambda_{j-1}'+i\]
where $\lambda'$ is the partition conjugate to $\lambda$. Note that every $r\in\{1,\dots,d\}$ can be written uniquely as $r=\lambda_1+\cdots+\lambda_{i-1}+j$ for some $i$ and $j$.

This permutation defines a map $s_\lambda\colon V^{\otimes d}\rightarrow V^{\otimes d}$ as follows:
\[s_\lambda(v_1\otimes\cdots\otimes v_d)=v_{\sigma_\lambda(1)}\otimes\cdots\otimes v_{\sigma_\lambda(d)}.\]

Denote by $\Lambda^dV$ the $d$-th exterior power of the module $V$. It is obtained from the $d$-th fold tensor product $V^{\otimes d}$ by taking the quotient with respect to the submodule
spanned by elements of the form $v \otimes v$, with $v \in V$. Similarly denote by $S^d V$ the symmetric power, obtained by taking the maximal quotient of $V^{\otimes d}$ on which $\frS_d$ acts trivially.

For $V\in\P_k$ the Schur module $S_\lambda V$ is defined as the image of the map
\[\Lambda^{\lambda'_1}V\otimes\cdots\otimes
\Lambda^{\lambda'_m}V\xrightarrow{\Delta\otimes\cdots\otimes\Delta} V^{\otimes d}\xrightarrow{s_{\lambda}}V^{\otimes d}\xrightarrow{\nabla\otimes\cdots\otimes\nabla}
S^{\lambda_1}V\otimes\cdots\otimes S^{\lambda_n}V,\]
where $\Delta$ resp. $\nabla$ is the inclusion resp. projection.

Note that $\Gamma^d=S_{(d)}$. In particular this means that we have
\[\Gamma^\lambda=S_{(\lambda_1)}\otimes\cdots\otimes S_{(\lambda_n)}.\]

In terms of strict polynomial functors, the decomposition (\ref{decomposition}) becomes
\[\Gamma^{\lambda} = \bigoplus_{\nu\in\Lambda(n,d)} K_{\mu \lambda} {S}_{\mu}.\]

\paragraph*{Tensor product}
The functor $\F$ sends the Schur functor $S_\lambda$ to the Specht module $\Sp(\lambda)$ (cf. \cite[6]{Green}), hence the tensor product of Schur functors is again given by the Kronecker coefficients:
\[{S}_{\lambda} \otimes_{\Gamma^d_k} {S}_{\mu} \cong \bigoplus_{\nu\in\Lambda(n,d)} g_{\lambda \mu}^{\nu} {S}_{\nu}\]

\subsection{From representations of the symmetric group to symmetric functions}
The characters of the symmetric group and the symmetric functions can be linked via the following \textit{characteristic map} (cf. \cite[I.~7]{Macd1995}):
\begin{align*}
\ch\colon R&\rightarrow \Lambda\\
\psi&\mapsto \left<\psi,\phi\right>_{\frS_d}
\end{align*}
where $\psi\in R^d$ and $\phi$ is the map sending an element $\pi\in\frS_d$ to $p_{\lambda_\pi}$ where $\lambda_\pi$ is the cycle-type of $\pi$.
\begin{proposition}[{\cite[I.~(7.~3)]{Macd1995}}] \label{characteristic}
The characteristic map is an isometric isomorphism of $R$ onto $\Lambda$.
\end{proposition}

If we denote by $[V]$ the character of the $k\frS_d$-module $V$, one gets for the Specht modules $\ch([\Sp(\lambda)])=s_{\lambda}$, in particular $\ch([\Sp((n))])=h_n$. Hence, for the character of the permutation module ${^\lambda M}$ we get
\[\ch([{^\lambda M}])=\ch([\Sp((\lambda_1))]\cdot \ldots \cdot[\Sp((\lambda_n))])=h_{\lambda_1}\cdot\ldots\cdot h_{\lambda_n}=h_\lambda.\]
  
\paragraph*{Kronecker product}
Using the characteristic map, one can define an internal product, sometimes called Kronecker product, via the internal tensor product of modules over the symmetric group.
For two symmetric functions $f=\ch(\phi)$ and $g=\ch(\psi)$ define 
\[f\ast g=\ch(\phi\cdot\psi).\]

\begin{example}
Since we know how to decompose the tensor product of two permutation modules, we can compute the Kronecker product of two complete symmetric functions:
\[h_{\lambda} \ast h_{\mu} = \ch([{^{\lambda}M}]\cdot[{^\mu}M])= \ch([{^\lambda M}\otimes_k {^\mu M}])=\ch([\bigoplus_{A\in A^\lambda_\mu}{^AM}])=\sum_{A \in A^{\lambda}_{\mu}} h_A\]
\end{example}

\subsection{From strict polynomial functors directly to symmetric functions}
There is an alternative description of the characteristic map, going directly from strict polynomial functors to symmetric functions.
Let $  \mathfrak{F} = \bigoplus_{d \geq 0} \Rep\Gamma_k^d$ be the category of strict polynomial functors of bounded degree.
Using the tensor product of the symmetric group modules, that corresponds for polynomial functors to the external tensor product, 
as defined in (\ref{exttensor}), one defines a product on the Grothendieck group $K(\mathfrak{F})$, which gives it the structure of a commutative, associative, graded ring with identity (cf. \cite{Macd1995}, Appendix~A) .

For $a=(a_1, \dots, a_n) \in k^n$, denote by $\diag(a)$ the diagonal endomorphism of $k^n$ with eigenvalues $(a_1, \dots, a_n)$. 
If $X$ is a polynomial functor, the trace of $X(\diag(a))$ is a polynomial function of $(a_1, \dots, a_n)$, which is symmetric. This determines a 
homomorphism of graded rings   \[\chi \colon K(\mathfrak{F})\rightarrow \Lambda, \] by \ $\chi(X)(a_1, \dots, a_n)=\trace X(\diag(a))$.

If we observe that $K(\Rep\Gamma_k^d) \cong K(\Mod k\frS_d) \cong R^d$ we may identify $K(\mathfrak{F})$ with $R$. Under this identification, the map $\chi$ coincides with $\ch$.
Hence we get (cf. also \cite[Appendix A]{Macd1995})
\[\chi({S}_{\lambda}) =s_{\lambda}.\]

Furthermore we can reobtain the correspondence between the strict polynomial functors $\Gamma^\lambda$ and the symmetric functions $h_\lambda$ by direct calculations:

We have \ $\Gamma^d(\diag(a)) = \diag(a) \otimes \dots \otimes \diag(a) = \diag(a)^{\otimes d}$, \ hence \ $\chi(\Gamma^d)(a_1, \dots, a_n)=\trace (\diag(a)^{\otimes d}) = (\trace(\diag(a)))^d
=(a_1 + \dots + a_n)^d$.

Recall that \[ (a_1 + \dots + a_n)^d= \sum \left( 
    \begin{array}{c}
      d \\
      m_1,\dots,m_n
    \end{array} \right)  a_1^{m_1}  \dots  a_n^{m_n}\] where the sum is taken over all compositions $(m_1, \dots, m_n)$ of $d$ and the coefficient of 
  $a_1^{m_1}  \dots  a_n^{m_n}$ \ equals \ $\frac{d!}{m_1! \dots m_n!}$. \ If we observe that this coefficient gives the number of permutations that 
  fix the partition $(m_1, \dots, m_n)$, we can rewrite the sum as 
  \[\sum_{|\lambda|=d} ~ \sum_{\beta\sim\lambda} a_1^{\beta_1}\cdots a_n^{\beta^n} = \sum_{|\lambda|=d} m_{\lambda}(a_1, \dots, a_n) = h_{\lambda}(a_1, \dots, a_n).\]
It follows that $\chi(\Gamma^d)=h_d$.
\bigskip  

From Theorem \ref{mainresult} and Proposition \ref{characteristic} we get
\begin{Corollary}
 The characteristic map $\chi$ sends the internal tensor product of strict polynomial functors to the Kronecker product of symmetric functions, i.e.
  \[ \chi(X \otimes_{\Gamma^d_k} Y) = \chi(X) \ast \chi(Y). \]
\end{Corollary}

In the following table we collect some of the correspondences we have shown before.
\begin{table}[ht]\centering \setlength{\tabcolsep}{8mm}
\begin{tabular}{ccc}\toprule
$\Rep\Gamma^d_k$&$k\frS_d \Mod $&$\Lambda$\\\midrule
$S_{(d)}=\Gamma^d$&$\Sp((d))={^{(d)}M}$&$s_d=h_{(d)}$\\\addlinespace[1ex]
$S_\lambda$&$\Sp(\lambda)$&$s_\lambda$\\\addlinespace[1ex]
$\Gamma^\lambda$&${^\lambda M}$&$h_\lambda$\\\addlinespace[1ex]
$\Gamma^{d,k^n}=\bigoplus\Gamma^\lambda$&$(k^n)^{\otimes d}=\bigoplus {^\lambda M}$&$\bigoplus h_\lambda$\\\addlinespace[1ex]
$X\otimes_{\Gamma^d_k}Y$&$\F(X)\otimes_{k}\F(Y)$&$\chi(X)\ast\chi(Y)$\\\bottomrule
\end{tabular}
\end{table}



\end{document}